\documentclass[12pt]{article}
\usepackage{amsmath,amsfonts,amssymb,amsthm,amscd}
\title{Dividing lines in unstable  theories and  subclasses of
Baire~1 functions}
\date{\today}
\author{Karim Khanaki\thanks{Partially supported by IPM grant 99030117}\\Arak University of Technology}

\newtheorem{Theorem}{Theorem}[section]
\newtheorem{Proposition}[Theorem]{Proposition}
\newtheorem{Definition}[Theorem]{Definition}
\newtheorem{Remark}[Theorem]{Remark}
\newtheorem{Lemma}[Theorem]{Lemma}
\newtheorem{Corollary}[Theorem]{Corollary}
\newtheorem{Fact}[Theorem]{Fact}
\newtheorem{Example}[Theorem]{Example}
\newtheorem{Question}[Theorem]{Question}

\begin{document}
\maketitle

\begin{abstract}  We give a new characterization of $SOP$ (the strict order property)
in terms of the behaviour of formulas in any model of the theory
as opposed to having to look at the behaviour of indiscernible
sequences inside saturated ones. We refine  a theorem of Shelah,
namely a theory has $OP$ (the order property)
 if and only if it has $IP$ (the independence property) or  $SOP$,
 in several ways by characterizing various
notions in functional analytic style.  We point out some
connections between dividing lines in first order theories and
subclasses of Baire 1 functions, and give new characterizations
of some classes and new classes of first order theories.
\medskip

\end{abstract}

\section{Introduction} \label{}
This paper aims to continue a new approach to Shelah stability
theory (in classical logic), which was followed in \cite{K3},
\cite{KP}. This approach is based on the fact that the study of
the model-theoretic properties of formulas in `models' instead of
only these properties in `theories' develops a sharper stability
theory and establishes important links between model theory and
other areas of mathematics, such as functional analysis. These
links lead to new results, in both model theory and functional
analysis, as well as  better understanding of the known results.

Let us give the background and our own point of view. In the 70's
Saharon Shelah developed local (formula-by-formula) stability
theory and combinatorial properties of formulas and used them to
gain global properties of theories. The independence property
 and the strict order property of a formula for a `theory' were
introduced in 1971 in \cite{Sh}. It is quite natural to try to
develop local  stability theory for formulas in  `models' instead
of only theories. Such a theory was developed in
\cite{Pillay-dimension}, \cite{Sh3}, \cite{BU} for the order
property and recently in \cite{K3} and \cite{KP} for the
independence property. In \cite{K3}, even a further step was
taken and the strict order property was studied and  a
connection  between a theorem of Shelah and an important theorem
in functional analysis was discovered.
  {\em What is interesting  is that some model-theoretic notions appeared
independently in topology and function theory, and moreover
various characterizations yield, via routine translations, the
characterization of $NOP/NIP/NSOP$   in a model $M$ or set $A$,
and some important theorems in model theory have twins there.}

Recall that  in \cite{Sh} Shelah introduced the strict order
property as complementary to the independence property: 

\medskip\noindent {\bf Shelah's Theorem\footnote{In this article, when we refer to Shelah's theorem, we mean this theorem.}:} (\cite[Theorem~II.4.7]{Shelah}) A complete first order theory has the order property ($OP$) if and only if it has the
independence property  ($IP$) or the strict order property
($SOP$).

\medskip\noindent
Later many classes
of independent $NSOP$ theories, such as simple and  $NSOP_n$,
were found.  In \cite{K3}, it is shown that there is a
correspondence between Shelah's theorem above and the well known
compactness theorem of Eberlein and \v{S}mulian.  In the current
paper, we complete some results of \cite{K3} and give a new
characterization of $SOP$ for classical logic.  In fact, the
correspondence mentioned above is completed in this article. {\em
What is substantial  is that there are  connections between
classification in model theory and classification of Baire class
1 functions  which lead to a better understanding of both of
these topics.}

It is worth recalling more historical points. Stability in a model
is not a new notion. In \cite{KM}, \cite{Pillay-dimension},
\cite{Sh3} and \cite{Iovino} this notion was studied in the
various contexts. (Although, the work of Krivine--Maurey
\cite{KM} is about the stability of the formula $\|x+y\|$ inside
a fixed Banach space and not the models of its theory.)  In
\cite{GL} some variants of $NOP/NIP/NSOP$ in a {\em type} were
defined and a local version of Shelah's theorem was proved.
Recently, in \cite{Ibarlucia}, \cite{K3} and \cite{Simon} the
connection between $NIP$ and functional analysis was noticed. The
notion ``$NIP$ of $\phi(x,y)$ in a model"  was introduced in
\cite{K3}.  We emphasize that our aim, approach and results in
\cite{K3}, \cite{KP} and the present paper are different from the
previous works.  In fact, the crucial idea in the paper is  to
study the  model theoretic properties of \textit{theories} by
\textit{studying model theoretic properties of formulas in
models.}

This paper is organized as follows. In the next section we first review some basic notions from functional analysis and  translate them into model theory.
We then  give a characterization of $SOP$
that does not involve indiscernible sequences and allows us to
relate the property to the behaviour of a class of Baire 1
functions (Proposition~\ref{NSOP=bounded variation} and
Remark~\ref{DBSC1} below).  We also refine Shelah's theorem 
(Theorem~\ref{strong NIP in a set} below) using a criteria for
formulas inside a model.
  We remark some equivalences on $NIP$ in the terms of function  spaces (Proposition~\ref{NIP=DBSC} and Remark~\ref{Chaatit})
   and define the  notion ``$NSOP$  in a model" (Definition~\ref{SOP in model}).
 In Section 3,  we point out connections between some dividing lines in
first order theories and subclasses of Baire class 1 functions (Remarks~\ref{Eberlein-Shelah}, \ref{SCP-NSOP} and Proposition~\ref{Shelah-like}).

\section{Model theory and function spaces}

We work in classical  ($\{0,1\}$-valued) model theory context,
although similar results are valid in the continuous logic
framework. Our model theory notation is standard, and a text such
as \cite{Shelah} will be sufficient background for the model
theory part of the paper. For the function theory part, read this
paper with \cite{K3} and \cite{KP} in your hand. We frequently
switch from model theory to function theory and vice versa, so we
provide some necessary functional analysis background.

First we recall some definitions and facts from functional
analysis  and then translate them into model theoretic language.

\subsection*{Function spaces}
We give definitions of the function spaces with which we shall be
concerned, with some of elementary relations between them.

Let  $X$ be a set and $A$ a subset of $\Bbb R^X$.  The topology of
{\em pointwise convergence} on $A$ is that inherited from the
usual product topology on $\Bbb R^X$; namely  the coarsest
topology on $A$ for which the map 
that sends each $f\in A$ to  $f(x)$
 is continuous for every $x\in X$.

Let $B$ be  some collection of real valued functions on $X$,
containing $A$. $A$ is said to be {\em  relatively compact (or
precompact) in $B$} if the closure $cl_B(A)$ of $A$ in $B$ is
compact. In this case $cl_B(A)$ is closed (and compact) in the
space $\Bbb R^X$, so in particular it implies that the closure of
$A$ in $\Bbb R^X$ is contained in $B$.

Recall that for a topological space $X$, $C(X)$ denotes the space
of all bounded continuous functions on $X$; it is a linear space
under pointwise addition.
 We can equip the space $C(X)$ with the {\em uniform norm topology},  the uniform metric defined by $d(f,g)=\sup_{x\in X}|f(x)-f(y)|$.
 The {\em  weak topology} on $C(X)$ is the coarsest topology such that every
 bounded linear functional on $C(X)$ is continuous.  So, $C(X)$ has three different topologies; namely
the topology of pointwise convergence, the uniform topology, and
the weak topology.

A well known fact in functional analysis states that for a compact
space $X$, the weak topology and the pointwise convergence
topology on norm-bounded subsets of $C(X)$ are the same. (See
Proposition 462E in \cite{Fremlin4}.)

 For a complete metric space $X$, a real-valued function $f$ on
$X$ is said to be of the first Baire class or Baire~1, if it is
the pointwise limit of a sequence $(f_n)$ of continuous functions
on $X$. This means that for each $\epsilon>0$ and each $x\in X$
there is a natural number $k$ such that $|f_n(x)-f(x)|<\epsilon$
for all $n\geq k$. The set of Baire 1 functions on $X$ is denoted
by $B_1(X)$.

 A real-valued function $f$ on a topological space $X$ is upper (resp.
lower) semi-continuous if and only if $\{x:f(x)\geq r\}$ (resp.
$\{x:f(x)\leq r\}$) is closed for every real number $r$. A
function $f$ is called semi-continuous if is either upper or lower
semi-continuous. A known classical theorem, due to Baire, asserts
that:

 \begin{Fact}[\cite{H}, p. 274]  \label{Baire} A real-valued function $f$ on
a complete metric space $X$ is lower semi-continuous if and only
if there is a sequence $(f_n)$ of continuous functions such that
$f_1\leq f_2\leq \cdots$ and $(f_n)$ converges pointwise to $f$
(for short we write $f_n\nearrow f$).
\end{Fact}

 A real-valued function $f$ on $X$ is called a difference of
 bounded semi-continuous functions  (short $DBSC$) if
there exist bounded semi-continuous functions $F_1$ and $F_2$ on
$X$ with $f=F_1-F_2$. The class of such  functions is denoted by
$DBSC(X)$. Since every lower (or upper) semi-continuous function
is the limit of a monotone sequence of continuous functions, so
$DBSC(X)\subseteq B_1(X)$. It is a well known fact that, in
general, $DBSC(X)$ is a {\em proper} subclass of $B_1(X)$ (see
Remark~\ref{Chaatit} below). To summarize,  $C(X)\subsetneqq
DBSC(X)\subsetneqq B_1(X)\subsetneqq {\Bbb R}^X$. (More details
can be found in \cite{CMR}.)

We will see shortly, the order property corresponds to  $C(X)$ and
the strict order property has connection to
 $DBSC(X)$.

In this paper, typically $A$ will be a subset of $C(X)$, the set
of bounded continuous functions on $X$. Moreover, it suffices to
assume that $A$ is countable and $X$ is compact and Polish, i.e. a
separable completely metrizable topological space. Note that the
uniform closure of $A$ is contained in $C(X)$ but in general the
poinwise closure of $A$ is not contained in $C(X)$, or even in
$B_1(X)$.

\subsection*{Model theory translation}
We fix an $L$-formula $\phi(x,y)$, a complete   $L$-theory $T$ and
a subset $A$ of the monster model of $T$. We let $\tilde\phi(y, x)
= \phi(x, y)$. Let $X=S_{\tilde{\phi}}(A)$ be the space of
complete $\tilde{\phi}$-types on $A$, namely the Stone space of
ultrafilters on Boolean algebra generated by formulas $\phi(a,y)$
for $a\in A$.  Each formula $\phi(a,y)$ for  $a\in A$ defines a
function $\phi(a,y):X\to\{0,1\}$, which takes $q\in X$ to 1 if
$\phi(a,y)\in q$ and to 0 if $\phi(a,y)\notin q$. Note that $X$
is compact and these functions are {\em continuous}, and as $\phi$
is fixed we can identify this set of functions with $A$.\footnote{We should be more careful with this assertion; if the variables $y$ are just dummy	variables that play no role, we can not recover $A$ from $X$. This result is true for the space of full types, but not necessarily for just the $\tilde\phi$-types. However, for the sake of simplicity we continue to write $A\subseteq C(X)$.}  So, $A$
is a subset of all bounded continuous functions on $X$, denoted by
$A\subseteq C(X)$.  Just as we did above, one can define $B_1(X)$
and $DBSC(X)$.

 To summarize, for an
$L$-formula $\phi(x,y)$ and a subset $A$ of an $L$-structure $M$,
we can assume that $A$ is a subset of $\Bbb R^X$ where
$X=S_{\tilde{\phi}}(A)$ and $A$ has the topology of pointwise
convergence as above. Moreover, every $f\in A$ is continuous,
i.e. $A\subseteq C(X)$.

The only additional thing we need to remark on is the following
result (see \cite[Corollary~2.10]{K3} and
\cite[Proposition~2.2]{Pillay-Grothendieck}):

\begin{Fact}[Eberlein--Grothendieck Criterion] \label{EG} Let $(a_i)$ be a sequence in some model of $T$ and $\phi(x,y)$
a formula. Then the  following are equivalent:

\noindent (i) There is no any sequence $(b_j)$ such that
$\phi(a_i,b_j)$ holds iff $i<j$.

\noindent (ii) For any sequence $(b_j)$,
$\lim_i\lim_j\phi(a_i,b_j)=\lim_j\lim_i\phi(a_i,b_j)$ when the
limits on both sides  exist.

\noindent (iii) Every function in the closure of
$\{\phi(a_i,y):S_{\tilde{\phi}}(\{a_i\})\to\{0,1\}:i<\omega\}$ is
continuous.
\end{Fact}

\noindent {\em History} ~ The equivalence
(ii)~$\Leftrightarrow$~(iii) in the general case, i.e. for
real-valued functions on arbitrary compact spaces, is due to
Grothendieck \cite{Gro}, which  he says it is based on an idea of
Eberlein. 
Pillay \cite{Pillay-dimension} proved  the equivalence
(i)~$\Leftrightarrow$~(iii) and pointed out that these conditions
are equivalent to definability of  coheirs (see
\cite{Pillay-Grothendieck}). In
\cite{Iovino}, Iovino also provided a proof of  Fact~\ref{EG} for real-valued formulas.  

\subsection{A new characterization of $SOP$}
First, we recall some notions and facts.
 Let $\phi(x,y)$ be formula and $n$ a natural number.  We say that a formula $\psi(x_1,\ldots,x_n)$ is a $\phi$-$n$-formula  if it is of the forms
$\exists y\big(\bigwedge_{i\in E}\phi(x_i,y)\wedge\bigwedge_{j\in F}\neg\phi(x_j,y)\big)$  or  $\forall y\big(\bigvee_{i\in E}\phi(x_i,y)\vee\bigvee_{j\in F}\neg\phi(x_j,y)\big)$ where $E,F$ are disjoint subsets of $\{1,\ldots,n\}$.  
In this case, $\psi(x_1,\ldots,x_n)$ has $n$ free variables $x_1,\ldots,x_n$ and a bounded variable $y$.
If $M$ is a model of a theory  and $\bar a=(a_1.\ldots,a_n)\in M^n$, the $\phi$-$n$-type of  $\bar a$, denoted by $tp_{\phi,n}(\bar a)$, is the set of all $\phi$-$n$-formulas $\psi(\bar x)$ such that $\models\psi(\bar a)$.\footnote{Although these  notions  seems very restrictive and
unnatural,  they are very useful for proving the main theorem of this section, i.e., Theorem~\ref{strong NIP in a set} below.
Note that the notion $\phi$-$n$-type is completely different from the notion $\phi$-type we defined earlier.}

\begin{Definition}[\cite{Shelah}, Definition~I.2.3]  Let $T$ be a complete $L$-theory, $\phi(x,y)$ an $L$-formula, $N$ a number
and $(a_i)$ a sequence in some model. The sequence $(a_i)$ is a
{\em $\phi$-$N$-indiscernible sequence} (over the empty set) if
for each $i_1<\cdots<i_N<\omega$, $j_1<\cdots<j_N<\omega$,
$$\textrm{tp}_{\phi,N}(a_{i_1}\ldots a_{i_N})=\textrm{tp}_{\phi,N}(a_{j_1}\ldots
a_{j_N}).$$
\end{Definition}

\begin{Fact} \label{Ramsey}  Let $T$ be a complete $L$-theory, $\phi(x,y)$ an $L$-formula, $M$ a model of $T$, and $N$ a natural
number.

 \noindent
(i)  If $(a_i)$ is an {\em infinite} sequence in $M$, there is an
infinite  subsequence $(b_i)$ which is a $\phi$-$N$-indiscernible
sequence.

 \noindent (ii)  If $I\subset J$ are two (infinite) linear ordered sets and
$(a_i)_{i\in I}$ is an {\em infinite}  $\phi$-$N$-indiscernible
sequence in $M$, there is a sequence $(b_j)_{j\in J}$ (possibly in
an elementary extension of $M$) which is a
$\phi$-$N$-indiscernible sequence and $(b_j)_{j\in J}$ has
the same $\phi$-$N$-type as $(a_i)_{i\in I}$.
\end{Fact}
\begin{proof} (i) follows from  (infinite)  Ramsey's theorem (see
  Theorem~I.2.4 of \cite{Shelah}) and (ii) follows from the
compactness theorem.
\end{proof}

\begin{Definition}[$SOP$ for a theory]  (i) Let $T$ be a complete $L$-theory, $\mathcal U$ the monster model of $T$, and $\phi(x,y)$ an
$L$-formula. We say that $\phi(x,y)$ has the {\em strict order
theory} (for the theory $T$) if there exists a sequence
$(a_i:i<\omega)$ such that for all $i<\omega$,  $$\phi(\mathcal U,
a_i)\subsetneqq\phi(\mathcal U, a_{i+1}).$$

\medskip\noindent
(ii) A complete theory $T$ has the {\em strict order property} if
there is a formula $\phi(x,y)$ which has the strict order
property (for $T$).

$SOP$ stands for the strict order property, and $NSOP$ for not the
strict order property.
\end{Definition}

In Definition~\ref{SOP in model} below we give a localized version of $SOP$.  (See also
Remark~\ref{remark sop in set} below.)

\medskip
As we will  see shortly, the following localized version of
Shelah's theorem leads to a new characterization of $SOP$ for a
theory. In the following theorem, we  will follow
the argument in Theorem 4.7, chapter 2 \cite{Shelah}.

\begin{Theorem}[Localized Shelah's theorem] \label{strong NIP in a set}  Let $T$ be a complete $L$-theory and $\phi(x,y)$ an
$L$-formula. Suppose that there are  infinite  sequences (not
necessarily indiscernible) $(a_i)$, $(b_j)$ in some model,  a
natural number $N$ and a set $E\subseteq \{1,\ldots,N\}$ such that
 \begin{itemize}
           \item [(i)]  for each $i_1<\cdots<i_N<\omega$,
$\psi(a_{i_1},\ldots,a_{i_N})$ holds,  where
$$\psi(x_1,\ldots,x_N):= \neg\Big(\exists y \big(\bigwedge_{i\in
E}\phi(x_i,y)\wedge \bigwedge_{i\in N\setminus
E}\neg\phi(x_i,y)\big)\Big), \text{ and }$$
           \item [(ii)]   $\phi(a_i,b_j)$ holds if and only if
           $i<j$.
   \end{itemize}
  \noindent Then the theory $T$ has $SOP$.
\end{Theorem}

 \noindent Before giving the proof let us remark:

\begin{Remark}
 (i) Note that  Theorem~\ref{strong NIP in a set}(i)
 identifies  a weaker condition ${\frak P}_{\phi,\bar{a}}$ than $NIP$ such that
$OP_{\phi,\bar{a},\bar{b}}+{\frak P}_{\phi,\bar{a}}$ implies
$SOP$, where $\bar{a}=(a_i),\bar{b}=(b_j)$ and
$OP_{\phi,\bar{a},\bar{b}}$ means that $\bar{a},\bar{b}$ witness
$\phi$ has the order property. We will see shortly, in fact, $SOP$
is  equivalent to the existence of $\bar{a},\bar{b}$ and $\phi$
such that $OP_{\phi,\bar{a},\bar{b}}+{\frak P}_{\phi,\bar{a}}$
holds. (See Proposition~\ref{NSOP=bounded variation} below.)

\noindent (ii) We will  establish a connection between this
presentation of $SOP$ and a well-known subclass of Baire~1
functions. (See Remark~\ref{DBSC1} below.) 

\end{Remark}

\begin{proof}[Proof of Theorem \ref{strong NIP in a set}]
 By Fact~\ref{Ramsey}, we can  assume that $(a_i)$ is a
 $\phi$-$N$-indiscernible sequence.  Now, we repeat the argument of Theorem~4.7, chapter~2  of \cite{Shelah}.
 By~(i), there are the natural number $N$ and $\eta : N \rightarrow
\{0,1\}$ defined by $\eta(i)=1$ if $i\in E$, and $=0$ otherwise, such
that $\bigwedge_{i\leq N} \phi(a_i,y)^{\eta(i)}$ is inconsistent.
(Recall  that for a formula $\varphi$, we use the notation
$\varphi^0$ to mean $\neg\varphi$ and $\varphi^1$ to mean
$\varphi$.) Starting with that formula, we change one by one
instances of $\neg\phi(a_i,y) \wedge \phi(a_{i+1},y)$ to
$\phi(a_i,y) \wedge \neg\phi(a_{i+1},y)$. Finally, we arrive at a
formula of the form $\bigwedge_{i<k} \phi(a_i,x) \wedge
\bigwedge_{k\leq i\leq N}
 \neg\phi(a_i,x)$. By (ii), the tuple $b_k$ satisfies that formula.
 Therefore, there is some
$i_0\leq N$, $\eta_0 : N \rightarrow \{0,1\}$ such that
$$\bigwedge_{i\neq i_0, i_0+1} \phi(a_i,y)^{\eta_0(i)} \wedge \neg\phi(a_{i_0},y) \wedge \phi(a_{i_0+1},y)$$
is inconsistent, but
$$\bigwedge_{i\neq i_0, i_0+1} \phi(a_i,y)^{\eta_0(i)} \wedge \phi(a_{i_0},y) \wedge \neg\phi(a_{i_0+1},y)$$
is consistent. Let us define $\varphi(\bar a,y)=\bigwedge_{i\neq
i_0,i_0+1} \phi(a_i,y)^{\eta_0(i)}$.  By Fact~\ref{Ramsey}, we
may increase the sequence $(a_i : i<\omega)$ to a
$\phi$-$N$-indiscernible sequence $(a_i:i\in \mathbb Q)$. Then for
$i_0 \leq i<i' \leq i_0+1$, the formula $\varphi(\bar a,y) \wedge
\phi(a_i,y) \wedge \neg\phi(a_{i'},y)$ is consistent, but
$\varphi(\bar a,y) \wedge \neg\phi(a_i,y) \wedge \phi(a_{i'},y)$
is inconsistent. Thus the formula $\theta(x,y) = \varphi(\bar a,y)
\wedge \phi(x,y)$ has the strict order property.
\end{proof}

Note that the formula $\theta(x,y)$ above has parameters. However
it is clear that if the formula $\eta(x,y,\bar a)$ has $SOP$,
where $\bar a$ are parameters, then so does the formula
$\eta(x,y\bar z)$.

\medskip
Now we want to establish a connection between $SOP$ and a class of
functions.  Recall that a real-valued function on a complete
metric space is said to be of the first Baire class, or Baire~1,
if it is the pointwise limit of a sequence of continuous
functions. The following lemma provides a connection  between
$SOP$ and a {\em proper} subclass of Baire~1 functions, namely
$DBSC$.

For easier reading, we note that the conditions (i), (ii) in
Lemma~\ref{bounded variation} below are abstractions of the notion
{\em  alternation number} in model theory. Of course, they are
not equivalent to the notion $NIP$ for a formula. (See the
explanations after Proposition~\ref{NSOP=bounded variation}
below.) It seems that the direction (i)~$\Rightarrow$~(iii) of
Lemma~\ref{bounded variation} is new to model theorists.

\begin{Lemma} \label{bounded variation}  Let $(f_n)$ be a sequence of $\{0,1\}$-valued functions
on a set $X$. Then the following are equivalent:

 \noindent (i) There  are a natural number $N$ and a set $E\subseteq
\{1,\ldots,N\}$  such that for each $i_1<\cdots<i_N<\omega$,

$$\bigcap_{j\in E}f_{i_j}^{-1}(1)\cap\bigcap_{j\in N\setminus E}f_{i_j}^{-1}(0)=\emptyset.$$

 \noindent (ii) There is a natural number $M$ such that
$\sum_1^\infty|f_n(x)-f_{n+1}(x)|\leq M$ for all $x\in X$.

 \noindent  Suppose moreover that $X$ is a compact metric  space and $f_n$'s are
continuous, then (ii) above (or equivalently (i)) implies (iii)
below:

 \noindent (iii) $(f_n)$ converges pointwise to a function  $f$ which is $DBSC$.
\end{Lemma}
\begin{proof} (i) $\Leftrightarrow$ (ii): Suppose that (i) holds.
Note that (i) states that we have a special {\em  pattern}   that
never exists; that is,  $\bigcap_{j\in E}f_{i_j}^{-1}(1)\cap\bigcap_{j\in N\setminus E}f_{i_j}^{-1}(0)=\emptyset$. Suppose, for a contradiction, that there is an element $x\in X$ such that $\sum_1^\infty|f_k(x)-f_{k+1}(x)|\geq 2N-1$. Therefore, there are $f_{k_1},\ldots,f_{k_{2N}}$ such that $(f_{k_i}(x)=1 \Leftrightarrow f_{k_{i+1}}(x)=0)$ for all $i<2N$. Let $i_1$ be  in $\{k_1, k_2\}$ such that the value of $f_{i_1}(x)$  is the same as the pattern above, i.e. $f_{i_1}(x)=1$ iff $1\in E$.
Let $i_2$ be in $\{k_3, k_4\}$ such that the value of $f_{i_2}(x)$  is the same as the pattern above. We can choose $i_3,\ldots,i_N$ similarly.  
Note that   $x\in\bigcap_{j\in E}f_{i_j}^{-1}(1)\cap\bigcap_{j\in N\setminus E}f_{i_j}^{-1}(0)$.  This is the special pattern above, a contradiction.
 The other direction is even easier. Indeed, let $N=M+2$, and
 $E=\{k: k\text{ is odd and } k\leq N\}$. Then $\bigcap_{j\in E}f_{i_j}^{-1}(1)\cap\bigcap_{j\in N\setminus E}f_{i_j}^{-1}(0)=\emptyset$.

(ii)~$\Rightarrow$~(iii): Clearly, $(f_n)$ converges pointwise to
a function $f$. (We can define $f_0(x)=0$ for all $x$.) Set  $F_1(x)=\sum_0^\infty(f_{n+1}-f_n)^+(x)$ and
$F_2(x)=\sum_0^\infty(f_{n+1}-f_n)^-(x)$. (Recall that for a
function $h:X\to\Bbb R$, $h^+(x)=\max(h(x),0)$ and
$h^-(x)=\max(-h(x),0)$.)  Then $f=F_1-F_2$ and $F_1, F_2$ are both
lower semi-continuous. (Note that
$g_k=\sum_0^{k-1}(f_{n+1}-f_n)^+\nearrow F_1$ and since the limit of
an {\em increasing} sequence of continuous functions is lower
semi-continuous (Fact~\ref{Baire}), so $F_1$ is lower
semi-continuous. Similarly for $F_2$.)
\end{proof}

\begin{Remark} (i) Note that Lemma~\ref{bounded variation}(i)  is an
abstraction of the condition (i) of Theorem~\ref{strong NIP in a
set}. Indeed, let $f_n(y)=\phi(a_n,y)$ where $a_n$ is a parameter
in some model.

\noindent (ii) Let us do a model theoretic translation, we set
$f_n(y)=\phi(a_n,y)$ and $X=S_{\tilde{\phi}}(A)$ where
$A=\{a_n:n<\omega\}\subseteq M\models T$. Clearly, $\phi(a_n,y)$
is continuous and since $A$ is countable, so $X$ is a {\em
metric} space.
This means that additional assumptions of (iii) in Lemma~\ref{bounded variation} hold.

\noindent (iii) We can expect a converse to
(ii)~$\Rightarrow$~(iii) of the  above lemma. Indeed, by
Fact~\ref{Baire} above, if $X$ is a compact metric space and $f$ is the $DBSC$ then there are a
sequence $(f_n)$ of (bounded) continuous functions and a natural
number $M$ such that $(f_n)$ converges pointwise to $f$ and
$\sum_1^\infty|f_n(x)-f_{n+1}(x)|\leq M$ for all $x$.

\noindent (iv) Note that Lemma~\ref{bounded variation}(ii)
guarantees that the sequence $(f_n)$ converges pointwise, but
there are Baire~1 functions which are not $DBSC$ (see
Remark~\ref{Chaatit}(i) below).
\end{Remark}

\medskip
The following gives a  new characterization of $SOP$ (for a
theory) and shows that the converse of Theorem~\ref{strong NIP in
a set} above is also true.

\begin{Proposition}[Characterization of $NSOP$] \label{NSOP=bounded
variation} Let $T$ be a  complete $L$-theory and $\mathcal U$ the
monster model of $T$. Then the following are equivalent:

 \noindent (i) $T$ is $NSOP$.

  \noindent (ii)  There are no formula $\phi(x,y)$ and sequences   $(a_i)$ and
$(b_j)$ in  $\mathcal U$, a natural number $N$ and a set
$E\subseteq\{1,\ldots,N\}$ such that two conditions (i) and (ii)
in Theorem~\ref{strong NIP in a set} hold, simultaneously.

  \noindent (iii) There are no formula $\phi(x,y)$ and \textbf{indiscernible} sequences   $(a_i)$ and
$(b_j)$ in  $\mathcal U$, a natural number $N$ and a set
$E\subseteq\{1,\ldots,N\}$ such that two conditions (i) and (ii)
in Theorem~\ref{strong NIP in a set} hold, simultaneously.

 \noindent (iv)  For any formula $\phi(x,y)$ and any
sequence (not necessarily indiscernible)  $(a_i:i<\omega)$, if
there is a natural number $N$ such that for any $b\in\mathcal U$,
$\sum_{i=1}^\infty|\phi(a_i,b)-\phi(a_{i+1},b)|\leq N$, then
there is no infinite sequence $(b_j)$ such that $\phi(a_i,b_j)$
holds iff $i<j$.

 \noindent (v) For any formula $\phi(x,y)$ and any \textbf{indiscernible}
sequence $(a_i:i<\omega)$, if there is a natural number $N$ such
that for any $b\in\mathcal U$,
$\sum_{i=1}^\infty|\phi(a_i,b)-\phi(a_{i+1},b)|\leq N$, then
there is no infinite sequence $(b_j)$ such that $\phi(a_i,b_j)$
holds iff $i<j$.

\medskip\noindent
Moreover, if $T$ is $NIP$ then $T$ is $NSOP$ iff for any formula
$\phi(x,y)$ there is a natural number $N$ such that for any
sequence  (not necessarily indiscernible) $(a_i:i<\omega)$, if for
any $b\in\mathcal U$,
$\sum_{i=1}^\infty|\phi(a_i,b)-\phi(a_{i+1},b)|\leq N$, then
there is no infinite sequence $(b_j)$ such that $\phi(a_i,b_j)$
holds iff $i<j$.
\end{Proposition}
\begin{proof}
 (i)~$\Rightarrow$~(ii) is Theorem~\ref{strong NIP in
a set}.  (ii)~$\Rightarrow$~(iii) is evident, and
(iii)~$\Rightarrow$~(ii) follows from Ramsey's theorem and the
compactness theorem.  (i)~$\Rightarrow$~(iv)  follows  from
Theorem~\ref{strong NIP in a set} and Lemma~\ref{bounded
variation}. (iv)~$\Rightarrow$~(ii) follows from
Lemma~\ref{bounded variation}.  (iv)~$\Rightarrow$~(v) is evident.

 (ii)~$\Rightarrow$~(i): Suppose, in order to get a
contradiction,  that  $\phi(x,y)$ has $SOP$ for the  theory  $T$. This
means that there is an \textbf{indiscernible} sequence $(a_i)$
such that $\exists y(\neg\phi(a_i,y)\wedge\phi(a_j,y))$ iff
$i<j$. So, there is some sequence $(b_j)$ such that
$\phi(a_i,b_j)$ holds iff $i<j$, i.e.,   the condition (ii) in
Theorem~\ref{strong NIP in a set} holds.  Let us define
$\psi(x_1,x_2)=\exists y(\phi(x_1,y)\wedge\neg\phi(x_2,y))$.
 So, for $i<j$,  $\psi(a_i,a_j)$ does not hold.  Let $N=\{1,2\}$ and
$E=\{1\}$ and $\psi(x_1,x_2)$ be as above.
 Then the condition (i) in Theorem~\ref{strong NIP in a
set} holds as well. This is a contradiction.

(v)~$\Rightarrow$~(i): By  Lemma~\ref{bounded variation} and an
argument similar to the direction (ii)~$\Rightarrow$~(i), the
proof is completed.
\end{proof}

Recall that for a formula $\phi(x,y)$, an indiscernible sequence
$(a_i)$ and a parameter $b$, the {\em alternation} of  $\phi(x,b)$
on $(a_i)$ is bounded by a natural number $n$, if there are at
most $n$ increasing indices $i_1<\cdots<i_n$ such that
$\models\phi(a_i,b)\leftrightarrow\neg\phi(a_{i+1},b)$ for all
$i<n$. A theory $T$ has $NIP$ if for any formula $\phi(x,y)$ there
is a natural number $n_\phi$ such that for any indiscernible
sequence $(a_i)$ and any parameter $b$, the  alternation of
$\phi(x,b)$ on $(a_i)$ is bounded by $n_\phi$. Note that in $NIP$
case, such numbers depend \textbf{just} on formulas.

Using this notion, Proposition~\ref{NSOP=bounded variation}(ii)
above asserts that a theory $T$ is $NSOP$ if for any formula
$\phi(x,y)$ and any  sequence $\bar a=(a_i)$, \textbf{if} there is
a natural number $n_{\phi,\bar a}$ such that for any $b$ the
alternation of $\phi(x,b)$ on  $\bar a$ is bounded by
$n_{\phi,\bar a}$, \textbf{then} there is no infinite sequence
$(b_j)$ such that $\phi(a_i,b_j)$ holds iff $i<j$.  Note that the
sequences are not necessarily indiscernible and such natural
numbers $n_{\phi,\bar a}$ depend on \textbf{both} the formulas
and the sequences; not just on formulas.
Thus,  Lemma~\ref{bounded variation} above  presents a `localized
 and wider' notion of alternation number.

\begin{Remark}   \label{DBSC1}
 Recall that,  for a set $A$ of an $L$-structure $M$  and  an $L$-formula $\phi(x,y)$, one can
consider the continuous function
$\phi(a,y):S_{\tilde\phi}(A)\to\{0,1\}$ defined by $\phi(a,q)=1$ if
$\phi(a,y)\in q$ and $0$ if $\phi(a,y)\notin q$. (Here
$\tilde{\phi}$ is the same formula as $\phi$, but we have
exchanged the role of variables and parameters, and
$S_{\tilde\phi}(A)$ is the space of  complete  $\tilde\phi$-types
over $A$.) If $A$ is countable, $S_{\tilde\phi}(A)$ is a compact
Polish space. Recall that, using a crucial result due to Eberlein
and Grothendieck (Fact~\ref{EG}), for a  sequence (not
necessarily indiscernible) $(a_i)$ there is no infinite sequence
$(b_j)$ such that $\phi(a_i,b_j)~\Leftrightarrow~i<j$ if and only
if every function in the pointwise closure of
$\{\phi(a_i,y):S_{\tilde\phi}(\{a_i\}_{i<\omega})\to\{0,1\}|~i<\omega\}$
is continuous.



\medskip  \noindent
  By Lemma~\ref{bounded variation} and
Proposition~\ref{NSOP=bounded variation}, $NSOP$ corresponds to
the class of functions which are difference of bounded
semi-continuous functions ($DBSC$) on the type spaces.  For a
formula $\phi(x,y)$ we set $D(\phi)=\{f:$  there exist $(a_i)$
and natural number $N$ such that $\phi(a_i,y)$ converges
pointwise to $f$ and
$\sum_1^\infty|\phi(a_i,q)-\phi(a_{i+1},q)|\leq N$ for all $q\in
S_{\tilde\phi}(\{a_i\}_{i<\omega})\}$. Similarly, we set
$C(\phi)=\{f:$  there exists  $(a_i)$ such that $\phi(a_i,y)$
converges pointwise to $f$  on
$S_{\tilde\phi}(\{a_i\}_{i<\omega})$ and $f$ is continuous$\}$.
Using these definitions and the facts above, a complete theory $T$
has $SOP$ if and only if there is a formula $\phi$ such that
$D(\phi)\setminus C(\phi)\neq\emptyset$. (See Fact~\ref{EG}
above.)  Notice that the above characterization of $NSOP$ is of
the form ``if~...~then~...". Indeed, by Fact~\ref{EG} and
Proposition~\ref{NSOP=bounded variation},  a theory $T$ is $NSOP$
if and only if

\medskip
``for any formula $\phi(x,y)$ and any (infinite) sequence
$(a_i:i<\omega)$, \textbf{if}  for some natural number $N$,
$\sum_1^\infty|\phi(a_i,q)-\phi(a_{i+1},q)|\leq N$ for all $q\in
S_{\tilde\phi}(\{a_i\}_{i<\omega})$, \textbf{then} $\phi(a_i,y)$
converges to a {\em continuous} function."
\end{Remark}

In \cite{KP}, the notions  $NIP$ and/or $NOP$ relative to a set or
model were studied.  We are now ready to introduce the analogous notion of $NSOP$ in a model or a set.

\begin{Definition}[$NSOP$ in a model]  \label{SOP in model}  {\em  Let $T$ be a complete $L$-theory, $\phi(x,y)$ an $L$-formula, and $M$ a model of $T$.
\newline
(i) A  set $\{a_{i}:i < \kappa\}$ of $l(y)$-tuples from $M$ is
said to be a   {\em $BSOP$-witness for $\phi(x,y)$} if the following
conditions (1),(2)  hold, simultaneously.
 \begin{itemize}
           \item [(1)] there are a natural number $N$ and a set
$E\subseteq\{1,\ldots,N\}$ such that for each
$i_1<\cdots<i_N<\kappa$, $M\models
\psi(a_{i_1},\ldots,a_{i_N})$   where
$$\psi(x_1,\ldots,x_N):= \neg\Big(\exists y \big(\bigwedge_{i\in
E}\phi(x_i,y)\wedge \bigwedge_{i\in N\setminus
E}\neg\phi(x_i,y)\big)\Big), \text{ and }$$
           \item [(2)]  for each natural number $n$ and $i_1<\cdots<i_n<\kappa$,
$$M\models \exists y_1\ldots y_n\Big(\bigwedge_{k<j\leq
n}\phi(a_{i_k},y_j)\wedge\bigwedge_{j'\leq k'\leq
n}\neg\phi(a_{i_{k'}},y_{j'})\Big).$$
 \end{itemize}
(ii) Let $A$ be a set of $l(x)$-tuples from $M$. Then   $\phi(x,y)$
 has $BSOP$-witness in $A$ if there is a countably infinite
sequence $(a_{i}:i<\omega)$ of elements of $A$ which is a
$BSOP$-witness for  $\phi(x,y)$.
\newline
(iii) Let $A$ be a  set of $l(x)$-tuples in $M$. We say that 
$\phi(x,y)$ has   $NBSOP$-witness in $A$ if it does
not have $BSOP$-witness in $A$.
\newline
(iv) $\phi(x,y)$ has  $NBSOP$-witness in $M$ if it
has  $NBSOP$-witness in the set of $l(x)$-tuples
from $M$.}
\end{Definition}

\begin{Remark} \label{remark sop in set} (i) If $\phi$ has $BSOP$-witness in $A$, then a Boolean
combination of instances of $\phi$ has $SOP$ for the theory $T$.
Of course, if $\phi$ has $SOP$ for $T$, then it has
$BSOP$-witness in some models of $T$.
\newline
(ii)  $\phi$ has $NSOP$ for the theory $T$ iff it has  $NBSOP$-witness in every model $M$ of $T$ iff it has $NBSOP$-witness in some model $M$ of  $T$ in which
all types over the empty set in countably many variables are
realised.
\newline
(iii)  If $\phi(x,y)$ has $BSOP$-witness in some model $M$ of
$T$, then there are arbitrarily long $BSOP$-witness for $\phi$
(of course in different models).
\end{Remark}

We will shortly give examples that indicate why  this notion is
useful (see Examples~\ref{example1} and \ref{example2} below).

\subsection{Remarks on $NIP$}
 We already knew that a theory is $NIP$ iff for any formula $\phi(x,y)$ and any sequence  $(a_i:i<\omega)$ in the monster model
 there is a subsequence $(a_{j_i}:i<\omega)$  such that for any element $b$ (in the monster model) there is
an eventual truth value of  $(\phi(a_{j_i},b):i<\omega)$. In the
language of function theory, the  subsequence
$(\phi(a_{j_i},y):i<\omega)$ converges to a (Baire 1) function
$f$. In the following we will see that the criterion presented in
Lemma~\ref{bounded variation} makes it possible to say more: the
limit $f$ should be $DBSC$.

\begin{Proposition}[Characterization of $NIP$] \label{NIP=DBSC} Let $T$ be a  complete $L$-theory, $\phi(x,y)$ an $L$-formula and $\mathcal U$
the monster model of $T$. Then the following are equivalent:

 \noindent (i) $\phi$ has  $NIP$ for $T$.

 \noindent (ii) For any
sequence (not necessarily indiscernible) $(a_i:i<\omega)$, there
is a subsequence $(a_{j_i}:i<\omega)$  such that for any
$b\in\mathcal U$ there is an eventual truth value of
$(\phi(a_{j_i},b):i<\omega)$.

 \noindent (iii) For any
sequence (not necessarily indiscernible) $(a_i:i<\omega)$, there
are a subsequence $(a_{j_i}:i<\omega)$ and a natural number $N$
such that
$\sum_{i=1}^\infty|\phi(a_{j_i},b)-\phi(a_{j_{i+1}},b)|\leq N$ for
\textbf{each} $b\in\mathcal U$.

\noindent (iv) For  any  sequence (not necessarily indiscernible)
$(a_i:i<\omega)$, there is a subsequence $(a_{j_i}:i<\omega)$
such that the sequence $\phi(a_{j_i},y)$ converges to a function
$f$ which is $DBSC$.
\end{Proposition}

\begin{proof} The equivalence (i)~$\Leftrightarrow$~(ii) is
folklore. The direction (iii)~$\Rightarrow$~(iv) follows from
Lemma~\ref{bounded variation}. The direction
(iv)~$\Rightarrow$~(ii) is evident.

\noindent (i)~$\Rightarrow$~(iii): Suppose, for a contradiction,
that  there is a sequence $(a_i)$  such that (iii) fails. Let $n$
be an arbitrary natural number and $\varphi$ be an arbitrary
formula. By Fact~\ref{Ramsey}, we can assume that $(a_i)$ is
$\varphi$-$n$-indiscernible.
 Then,  by Lemma~\ref{bounded variation},  there is a (finite) subsequence $a_{j_1},\ldots,a_{j_n}$ and
$b\in\mathcal U$ such that $\phi(a_{j_i},b)$ holds iff $i$ is
even.  As $n$ and $\varphi$ are arbitrary,  the following set is
a type
$$\Big\{\exists
y\big(\bigwedge_{i=1}^n\phi(x_{2i},y)\wedge\neg\phi(x_{2i+1},y)\big)
\wedge\big((x_1,\ldots,x_{2n+1})\text{ is
indiscernible}\big):n<\omega\Big\}.$$ By the compactness theorem,
there are an indiscernible sequence $(c_i)$ and an element  $d$
such that $\phi(c_i,d)$ holds if and only if $i$ is even,  a
contradiction.
\end{proof}

Proposition~\ref{NIP=DBSC}(iv) above  and the following remark
show that the approach of the present paper would be useful.

\begin{Remark} \label{Chaatit} (i) Recall that for a compact metric space $X$ and a subset $A\subset X$,
then the indicator function $1_A$ is Baire 1 if and only if $A$ is
both $F_\sigma$ and $G_\delta$.
(See  Definition~24.1 and Theorem~24.10 of \cite{Kechris}. In this case, notice that the sets $A=\{x:1_A(x)=1\}$ and $A^c=\{x:1_A(x)=0\}$ are both $G_\delta$; that is, they are  countable intersections of open sets.)
 Notice that the class of functions which are difference of bounded semi-continuous
functions is a {\em  proper} subclass of Baire 1 functions.
Furthermore, every $\{0,1\}$-valued function is the $DBSC$ if and
only if there exist  disjoint differences  of closed sets
$W_1,\ldots,W_m$ such that  $f=\sum_{i=1}^m 1_{W_i}$  (see
\cite[Proposition~2.2]{CMR}). This result makes clear why
$DBSC(X)$ is a {\em proper} subclass of $B_1(X)$.

\noindent (ii) As Pierre Simon pointed out to us, it is  known
that  for every sequence $(a_i)$ in the monster model of a $NIP$
theory one can find a subsequence $(a_{j_i})$ that their types
converges to a finitely satisfiable type and it is  known that
invariant types in $NIP$ theories  have definitions which are
finite Boolean combinations of closed sets (see
\cite[Proposition~2.6]{HP}). In fact, by the above remark, that
is equivalent to Proposition~\ref{NIP=DBSC}(iv).
\end{Remark}

The following statement clearly indicates why some people-- not
all them-- say that the independence property and the strict order
property are {\em orthogonal}. That is, Shelah's theorem is of
the form  ``$p\wedge(p\to s)\equiv$ stability."

\begin{Corollary}[Shelah's Theorem, revisited]
  \label{Shelah revisited}
Let $T$ be a complete $L$-theory. Then the following are
equivalent:
\begin{itemize}
  \item [(1)] $T$ is stable.
  \item [(2)] The following two properties hold:
   \begin{itemize}
           \item [(i)]($NIP$): For any formula $\phi(x,y)$ and any
 sequence (not necessarily indiscernible) $(a_i:i<\omega)$, there are a subsequence
$(a_{j_i}:i<\omega)$ and a natural number $N$ such that the
sequence $\phi(a_{j_i},y)$ converges to a function $f$ and  for
any $b$ in the monster model,
$\sum_1^\infty|\phi(a_{j_i},b)-\phi(a_{j_{i+1}},b)|\leq N$,  and
           \item [(ii)]($NSOP$):   For any formula $\phi(x,y)$ and any
sequence (not necessarily indiscernible) $(a_i:i<\omega)$,
\textbf{if}  the sequence $\phi(a_i,y)$ converges to a function
$f$ and there is some natural number $N$ such that for any $b$ in
the monster model,
$\sum_1^\infty|\phi(a_{i},b)-\phi(a_{i+1},b)|\leq N$,
\textbf{then} $f$ is continuous.
   \end{itemize}
\end{itemize}
\end{Corollary}
\begin{proof} By Proposition~\ref{NIP=DBSC}, $NIP$ is equivalent to
(i). By Proposition~\ref{NSOP=bounded variation} and Fact~\ref{EG}
(or just Remark~\ref{DBSC1}), $NSOP$ is equivalent to (ii). Now,
by the usual form of Shelah's theorem the proof is completed.
\end{proof}

We can give a proof of Shelah's theorem above using a well-known
theorem of functional analysis, namely the Eberlein--\v{S}mulian
Theorem (Fact~\ref{Eberlin-Smulian} below). Also, one can provide
a local version of Shelah's theorem: A formula $\phi$ is stable
for the theory $T$ iff the conditions (i),(ii) above hold for
$\phi$.  We will compare shortly the above observations with  the
Eberlein--\v{S}mulian Theorem.

\subsection{Examples}
To clarify the results, we  build some
examples. First, we give a model $M$ and a formula $\phi(x,y)$
such that $\phi$ has $OP$ and $NIP$ in $M$, and $Th(M)$ has $SOP$. 
This example is not interesting in itself but it is a step
towards an example with interesting properties.

\begin{Example}  \label{example1} Let $A=\{a_i:i<\omega\}$ and $B=\{b_i:i<\omega\}$.
We define a binary relation $R(x,y)$ on $D=(A\cup B)\times (A\cup
B)$ as follows:

\noindent (1) $R(a_i,a_j)$ holds iff $i<j$,

\noindent  (2) For each $k<\omega$, we define:

  ~(2--k)~~   for any $i\leq k$, $R(a_{2^k+i},b_k)$ holds iff $i$ is even, and
for any $j<2^k$ or $j>2^k+k$, $\neg R(a_j,b_k)$ holds.

\noindent (3) For any other $(c,d)\in D$, $R(c,d)$ does not hold.

\medskip\noindent
(Note that (1) says that $R(x,y)$ has the order in $M=A\cup B$. It
is easy to verify that  the formula $R(x,y)$ is $NIP$ in $M$.)

\medskip\noindent
Moreover, $Th(M)$ has  $SOP$. Indeed, notice that $\neg R(a_i,{\cal U})\subsetneqq \neg R(a_{j},{\cal U})$ for all odd numbers $i<j$.
\end{Example}

In the following, we give a model $N$ and a formula $\phi(x,y)$
such that $\phi$ has $NIP$ and $OP$ in $N$, and moreover $Th(N)$
has $IP$ and $\phi$ has $NSOP$ for $Th(N)$.

\begin{Example} \label{example2} Let $A=\{a_i:i<\omega\}$ and for any infinite subsequence $I$ of $\omega$, let $B_I=\{b_i^I:i<\omega\}$.
We define a binary relation $R(x,y)$ on $D=(A\cup \bigcup_I
B_I)\times (A\cup \bigcup_I B_I)$ as follows:

\noindent (1) $R(a_i,a_j)$ holds iff $i<j$,

\noindent  (2) For each infinite subset $I$ of $\omega$
and each $k<\omega$, the condition (2--k) in the above example
holds for $A=\{a_i:i\in I\}$ and $B=B_I$.

\noindent (3) For any other $(c,d)\in D$, $R(c,d)$ does not hold.

\medskip\noindent
Now, it is easy to verify that  the formula $R(x,y)$ is $NIP$ in
$N=A\cup \bigcup_I B_I$ but it has $OP$ in $N$. Also, (2)
guarantees that the complete theory of this structure has $IP$ (see
Proposition~\ref{NIP=DBSC}(iii)). But, by Lemma~\ref{bounded
variation}, one can show that its theory does not have $SOP$. In
fact, the type of $SOP$ (for any formula) is not consistent with
$Th(N)$. 
Indeed, notice that for any natural number $N$, there is some
natural number $m$ such that there is \textbf{no}  any subsequence
$c_1,\ldots,c_m$ of $(a_i)$ such that for each $i_1<\cdots<i_N\leq
m$, $\models\psi(c_{i_1},\ldots,c_{i_N})$, where
$\psi(x_1,\ldots,x_N)$ is the formula in Theorem~\ref{strong NIP
in a set}(i) with $\phi(x,y)=R(x,y)$ (or any other formula).

\medskip\noindent
 This example confirms that there is
a formula $\phi$  $NSOP$ for a  theory and a sequence $(a_i)$ such
that the sequence $(\phi(a_i,y):i<\omega)$ pointwise converges to
a non-continuous function. This statement contrasts with the
theory of Random Graph (see Example~\ref{example3} below).
\end{Example}

\section{Dividing lines in model theory and  Baire class~1 functions}
 This part is mainly expository but is (in our view) very illuminating.
We point out some parallels between  model theoretic dividing
lines for first order theories and   subclasses of Baire~1
functions, and propose a  new thesis. For this, we recall some
notions and the following well-known theorem of functional
analysis.

 If $X$ is a
topological space then $C(X)$ denotes the space of bounded
continuous functions on $X$. A subset $A\subseteq C(X)$ is {\em
relatively  weakly (pointwise) compact} if it has compact closure
in the weak (pointwise) topology on $C(X)$. Notice that for a
compact space $X$, a subset $A$ of $C(X)$ is weakly compact if and only if
it is norm-bounded and pointwise compact
(cf. \cite[Theorem~462E(ii)]{Fremlin4}).

\begin{Fact}[Eberlein--\v{S}mulian Theorem]
\label{Eberlin-Smulian} Let $X$ be a compact Hausdorff space and
 $A$ a norm-bounded subset of  $C(X)$. Then for the topology of pointwise convergence the following are equivalent:
\begin{itemize}
  \item [(1)]   $A$  is   relatively compact in $C(X)$.
  \item [(2)] The following two properties hold:
   \begin{itemize}
           \item [(i)]($RSC$) $A$ is relatively sequentially compact in ${\Bbb R}^X$, and 
           \item [(ii)]($SCP$)  $A$ has the sequential completeness property. 
   \end{itemize}
\end{itemize}
\end{Fact}

\noindent {\em Explanation.}   See \cite{W} for a proof of the
Eberlein--\v{S}mulian theorem.  Recall that, $A$ is relatively sequentially compact in ${\Bbb R}^X$ if every sequence of $A$ has a convergent subsequence in ${\Bbb R}^X$, and $A$ has the sequential completeness property if  the limit of every convergent sequence of $A$ is continuous.  (2) is precisely  the condition $B$
in the main theorem of \cite{W}. Indeed, each sequence  contains
a  subsequence converging to an element of $C(X)$ if and only if
(i)  each sequence has a convergent subsequence in ${\Bbb R}^X$,
and (ii)  the limit of every convergent sequence is continuous.
Also, (1)  is the condition $A$ in \cite{W}.

 \begin{Remark}  \label{Eberlein-Shelah}  Recall from \cite{Ros}  (or Proposition~\ref{NIP=DBSC} above) that $NIP$ implies 2.(i) in the Eberlein--\v{S}mulian Theorem. By Proposition~4.6 of \cite{K3} (or Remark~\ref{SCP-NSOP} below),
  if for every countable set $A$ of the monster model and every formula $\phi$
the condition 2.(ii) holds, then the theory is $NSOP$. Notice
that, by Proposition~\ref{NSOP=bounded variation} (or
Remark~\ref{DBSC1}) and  Proposition~\ref{NIP=DBSC}, the converses
do not hold. Recall from \cite{K3} that 2.(ii) is called the weak
sequential completeness property (short $SCP$), and 2.(i) is
called the relative sequential compactness (short $RSC$). Notice
that relative compactness of $A$ corresponds to stability, by a
criterion due to Eberlein and Grothendieck (Fact~\ref{EG}). Now,
we can complete the diagram presented in \cite{K3}:

\medskip

~~~~~~~~~~~~~~~~~~~~~~~~~~~~ {\scriptsize Shelah}

~~ Stable ~~~~~~~~~~~~~~~~~ $\Longleftrightarrow$ ~~~~~~~~~~~~~
$NIP$ ~~~~~~~ $\&$ ~~~~~ $NSOP$

\bigskip

~~~~~ $\Updownarrow$ {\scriptsize Eberlein--Grothendieck}
~~~~~~~~~~~~~~~~~~~~~ $\Downarrow$
 {\scriptsize
 \ \ \ \ \ \ \ \ \ \ \ }
  ~~~~~~~~~~~~ $\Uparrow $

\medskip

~~~~~~~~~~~~~~~~~~~~~~~~ {\scriptsize Eberlein--\v{S}mulian}

 Weak Compactness ~~~~ $\Longleftrightarrow$
~~~~~~~~~~~~~ RSC ~~~~~~~ $\&$ ~~~~~~ $SCP$

\bigskip

\medskip

We will shortly prove that $NIP$ is equivalent to $RSC$ under
compactness, and $SCP$ and $NSOP$ are not equivalent. We  make a
point that $NIP$ together with  many conditions correspond to stability, so of course
there is  no reason to expect that all notions agree.
\end{Remark}


\subsection*{A thesis}

 In the Eberlein--\v{S}mulian Theorem, notice that (ii) is the
weakest topological property such that (i) and (ii) imply
relative compactness. This leads to the following definition.

\begin{Definition}  \label{Baire 1}  Let $T$ be a complete $L$-theory. We say
that  $T$ has

\medskip\noindent (i)  the  relative sequential compactness property (short $RSC$) if

\medskip
 $(RSC)$ \ \ \ for any formula  $\phi(x,y)$ and any infinite sequence $(a_i:i<\omega)$, there is a subsequence
$(a_{j_i}:i<\omega)$  such that for any parameter $b$ there is an
eventual truth value of  $(\phi(a_{j_i},b):i<\omega)$.

\medskip\noindent (ii)  the  sequential completeness property (short $SCP$) if

\medskip
 $(SCP)$ \ \ \ for any formula  $\phi(x,y)$ and any infinite sequence $(a_i:i<\omega)$, \textbf{if}
for every $b$ in the monster model there is an eventual truth
value of the sequence $(\phi(a_i,b):i<\omega)$, \textbf{then}
there is no infinite sequence $(b_j:j<\omega)$ such that
$\phi(a_i,b_j)$ holds iff $i<j$.

\end{Definition}

\begin{Remark} \label{SCP-NSOP}  (i) Every stable theory has the  $SCP$.

\noindent (ii) A theory is $NSOP$ if it has the $SCP$.

\noindent (iii) A theory is $NIP$ if and only if it has $RSC$.

\noindent (iv) A theory is stable if and only if it is $NIP$ and
has the  $SCP$.
\end{Remark}
\begin{proof} (i): Immedaite. 

(ii): Suppose that there are sequences $(a_i),(b_j)$ and formula
$\phi(x,y)$ such that the conditions of Theorem~\ref{strong NIP in
a set} hold. (Equivalently, the theory is $SOP$.) Then
$\phi(a_i,y)$ converges to a function $f$ which is not continuous.
(See Fact~\ref{EG}.) So, the $SCP$ fails.

(iii): This is the  equivalence (i)~$\Leftrightarrow$~(ii) of
Proposition~\ref{NIP=DBSC}. (Note that in {\em function theory}
the condition (iv) of Proposition~\ref{NIP=DBSC} (equivalently
$NIP$) {\em strictly} implies $RSC$, but their equivalence in
model theory is due to compactness theorem.)

(iv): By (ii) above, $SCP$ implies $NSOP$. So, $NIP$ and $SCP$
imply stability, by Shelah's theorem. (One can give a proof using
the Eberlein--\v{S}mulian Theorem and Fact~\ref{EG}.)
 The converse is evident.
\end{proof}

\begin{Example} \label{example3} (i) The theory of Random Graph has the  $SCP$.
Indeed, for any formula $\phi(x,y)$, either  $\phi(x,y)$ is stable, or
there is  no infinite
sequence $(a_i:i<\omega)$ such that for any $b$ in the monster
model there is an eventual truth  value of the sequence
$(\phi(a_i,b):i<\omega)$. Recall that the theory of Random Graph has quantifier elimination. Therefore, one can easily check that atomic formulas are either stable, or satisfy  the second alternative. By quantifier elimination, this holds for every formula.

\noindent (ii) The theory in Example~\ref{example2} is $NSOP$ but
it does not have the $SCP$.
\end{Example}

\noindent Note that in Lemma~\ref{bounded variation} we did not
give a converse to (ii)~$\Rightarrow$~(iii). This suggests the
following definition: A complete theory $T$ has the $DBSC$ if and
only if for any formula  $\phi(x,y)$ and any infinite sequence
$(a_i:i<\omega)$, \textbf{if}
 the sequence $(\phi(a_i,y):i<\omega)$ converges to a function which is
$DBSC$, \textbf{then} there is no infinite sequence
$(b_j:j<\omega)$ such that $\phi(a_i,b_j)$ holds iff $i<j$.
Clearly, if  a theory is $DBSC$ then it is $NSOP$. Now we want to
continue this process to create a hierarchy  of theories. Let
$\mathcal C$ be some subclass of Baire~1 functions, containing
$DBSC$. We say that a theory $T$  is (or has) $\mathcal C$ if

\medskip
  \ \ \ for any formula  $\phi(x,y)$ and any infinite sequence $(a_i:i<\omega)$, \textbf{if}
 the sequence $(\phi(a_i,y):i<\omega)$ converges to a function which is
$\mathcal C$, \textbf{then} there is no infinite sequence
$(b_j:j<\omega)$ such that $\phi(a_i,b_j)$ holds iff $i<j$.
\medskip

Now we can give other Shelah-like theorems.

\begin{Proposition}  \label{Shelah-like}  Let $T$ be complete theory and $\cal C$ as
above. Then $T$ is stable if and only if it is both $NIP$ and
$\cal C$.
\end{Proposition}
\begin{proof}
 The proof is
similar to the argument of Remark~\ref{SCP-NSOP}(iv). (Note that
$\cal C$ implies $DBSC$ and so $NSOP$.)
\end{proof}

\medskip

 \noindent
Notice that the $SCP$ ($DBSC$) asserts that for any formula
$\phi(x,y)$, every Baire~1 ($DBSC$) function in the closure of
$\phi(a,y)$'s is continuous. Set $Baire~1(\phi)=\{f:$ there
exists  $(a_i)$ such that $\phi(a_i,y)$ converges pointwise to
$f\}$, $DBSC(\phi)=\{f:$ there exists  $(a_i)$ such that
$\phi(a_i,y)$ converges pointwise to $f$ and $f$ is $DBSC \}$ and
$C(\phi)=\{f:$ there exists $(a_i)$ such that $\phi(a_i,y)$
converges uniformly to $f\}$ as Remark~\ref{DBSC1} above. (Notice the difference
between $DBSC(\phi)$ and $D(\phi)$ in Remark~\ref{DBSC1}.)
 By these notations, we say
that $T$ has Baire~1 property (equivalently the $SCP$)  iff for
any formula $\phi$, $Baire~1(\phi)\setminus C(\phi)=\emptyset$.
Similarly, we say that $T$  is (or has) $DBSC$ iff for any formula
$\phi$, $DBSC(\phi)\setminus C(\phi)=\emptyset$. We can do this
process for each subclass $\mathcal C\supseteq DBSC$ of Baire~1
functions in the sense of \cite{KL}; a theory is (or has)
$\mathcal C$ iff for any formula $\phi$, $\mathcal C
(\phi)\setminus C(\phi)=\emptyset$.


\medskip\noindent
\textbf{ Notation:}  In the rest of this part, the symbol $\frak
P$ (generated by  frak\{P\}) denotes an arbitrary model theoretic
property such that $\frak P$ implies  $NSOP$ (for
example, $NSOP_n$ or simplicity), and $\mathcal C$ denotes an
arbitrary subclass of Baire~1 functions on compact Polish spaces,
containing $DBSC$. For a theory $T$ and a formula $\phi(x,y)$ we
set $\mathcal C(\phi)=\{f:$ there exist $(a_i)$ such that
$\phi(a_i,y)$ converges pointwise to $f$ and $f\in\mathcal C\}$,
and $C(\phi)$ as Remark~\ref{DBSC1} above.
 For a model theoretic
property $\frak P$, if there is a subclass $\mathcal C$ of Baire~1
functions such that any theory $T$ is $\frak P$ if and only if
for any formula $\phi$, $\mathcal C(\phi)\setminus
C(\phi)=\emptyset$, then we write
  $\frak P=\frak P_\mathcal C$.
Similarly, for a subclass $\mathcal C$, if there is a model
theoretic property $\frak P$ such that any theory $T$ has $\frak
P$ if and only if for any formula $\phi$, $\mathcal
C(\phi)\setminus C(\phi)=\emptyset$, then we write $\mathcal
C=\mathcal C_\frak P$.

\medskip

 Recall that $DBSC$ implies $NSOP$, and stability (or $NOP$) corresponds to
the class of continuous functions (short Continuous). With these
notations, $\frak P_{DBSC}\subset_{_?} NSOP$, $\frak
P_{\text{Continuous}}=NOP$ and $\mathcal C_{NSOP}\subset_{_?}
DBSC$, $\mathcal C_{NOP}=\text{Continuous}$. Now, one can suggest
the following diagram:

$$NOP=\frak P_{\text{Continuous}}\varsubsetneqq\cdots\varsubsetneqq \ ?_{\text{Baire~1}}\varsubsetneqq\cdots\varsubsetneqq\frak P\subsetneqq\dots\varsubsetneqq\frak P_{\text{$DBSC$}}\subset_{_?}NSOP$$
$$\text{Baire 1}=\mathcal C_{?}\supsetneqq \cdots\supsetneqq\mathcal C\supsetneqq\cdots\supsetneqq \text{$DBSC$} \ _{_?}\supset\mathcal C_{NSOP}\supsetneqq\mathcal C_{NOP}=\text{Continuous}$$

\bigskip

There are so many questions: for a model theoretic property
$\frak P$, what is the right class $\mathcal C_{\frak P}$? And
converse, for a subclass $\mathcal C$, what is the right model
theoretic property $\frak P_\mathcal C$?

Let us discuss possible answers. There are four possibilities.
First: there are correspondences between some model theoretic
classes and subclasses of Baire~1 functions. (See Question~\ref{question} below.) Second: some model
theoretic dividing lines imply some subclasses of Baire~1
functions, or  vice versa. Third: some model theoretic classes
are divided by some subclasses of Baire~1 functions, or vice
versa. Fourth: there are connections between subclasses of Baire~1
functions and classes in   Keisler's order. Everything that is the
case is good.

\begin{Question} \label{question}
 Are there any interesting relations between subclasses of Baire~1 functions and notions like $NSOP_n$?
\end{Question}

Finally, we point out that the notion $NSOP$ says that  \textbf{if}  any sequence of the form
$\phi(a_n,y)$ converges with a  `special rate', \textbf{then} the
limit is continuous. One can expect other properties  also have
the same nature. If that is the case, the special rate for $NSOP$
is  {\em stronger}  than the special rate for $\frak P$.   The
above points strongly inspire us to believe that
 model theoretic  classification
is correlated with  a classification of Baire~class~1 functions
similar to the work of Kechris and Louveau in  \cite{KL}.

\bigskip\noindent
 {\bf Acknowledgements.}
I am very much indebted to Professor John T. Baldwin  for his
kindness and his helpful comments.  I want to thank Pierre Simon
for his interest in reading a preliminary version of this article
and for his comments.  I thank the anonymous referee for his/her detailed
suggestions and corrections; they helped to improve significantly the exposition of this paper.

 I would like to thank the Institute for Basic Sciences (IPM), Tehran, Iran. Research partially supported by IPM grant  no  99030117.

\end{document}